\documentclass[12pt,reqno]{amsart}
		\usepackage[hyphens]{url} 
		\numberwithin{equation}{section}
		
		\makeatletter
		\@namedef{subjclassname@2020}{%
			\textup{2020} Mathematics Subject Classification}
		\makeatother
		
		\usepackage{amsmath,amssymb,color}
		\usepackage{amsthm} %theorem environment option
		\usepackage{multicol}
		\usepackage{mathtools}
		\usepackage{graphicx}
		\usepackage{mathptmx}
		\usepackage[bookmarksnumbered,colorlinks]{hyperref}
		\mathtoolsset{showonlyrefs}
		\usepackage[pagewise,mathlines]{lineno}%\linenumbers
		\usepackage{enumitem}
		\usepackage{soul}
		
		\theoremstyle{plain} %text of this environment is typesetted in italics
		\newtheorem{theorem}{\indent\sc Theorem}[section]
		\newtheorem{lemma}[theorem]{\indent\sc Lemma}
		\newtheorem{corollary}[theorem]{\indent\sc Corollary}

		\theoremstyle{definition} %text of this environment is typesetted in roman letters
		
		\newtheorem{remark}[theorem]{\indent\sc Remark}

		\title[A Gage-type estimate for the $p$-fundamental tone]{\sc A Gage-type estimate for the $p$-fundamental tone and applications to submanifolds}
		
		\bigskip
		
		\author[F.R. dos Santos and M.N. Soares]{F\'abio R. dos Santos$^{\ast}$ and Matheus N. Soares}
		
		\address{% First Author
			Departamento de Matem\'atica \\
			Universidade Federal de Pernambuco \\
			50.740-540 Recife, Pernambuco \\
			Brazil}
		\email{fabio.reis@ufpe.br}
\email{matheus.nsoares@ufpe.br}
		
		
\thanks{$^{\ast}$Corresponding author}
		
\begin{document}
			
\begin{abstract}
A generalization of the classical Gage's upper bound for the fundamental tone of the Laplacian on complete non-compact Riemannian manifolds to the $p$-Laplacian context is obtained. As an application, upper estimates for the $p$-fundamental tone of complete non-compact submanifolds of space forms with constant nonpositive sectional curvature, and for a class of product spaces, are presented.
\end{abstract}
			
\maketitle
		
\section{Introduction and preliminaries}

In the last few decades, the interest in the study of the first eigenvalue of certain elliptic operators has increased substantially. Among these, it is worth highlighting the $p$-Laplacian operator, which arises as a natural extension of the Beltrami-Laplacian operator. More precisely, on a Riemannian manifold $\Sigma^{n}$, we define the $p$-Laplacian operator by
\begin{equation}\label{eq_2.1}
\Delta_{p}u={\rm div}(|\nabla u|^{p-2}\nabla u),\quad1<p<\infty,
\end{equation}
for any function $u\in W^{1,p}_{loc}(\Sigma)$. The $p$-Laplacian appears naturally on the variational problems associated with the energy functional $E_{p}:W^{1,p}_{0}(\Sigma)\to\mathbb{R}$ given by
\begin{equation}\label{eq_2.2}
E_{p}(u)=\int_{\Sigma}|\nabla u|^{p}d\Sigma.
\end{equation}
In particular, if $p=2$, the $p$-Laplacian $\Delta_{p}$ is the usual Laplace operator $\Delta$. Given a bounded domain $\Omega$ in $\Sigma^{n}$, the first eigenvalue of the Dirichlet problem for the $p$-Laplacian operator on $\Omega$ can be characterized variationally as
\begin{equation}\label{eq_2.3}
\lambda_{1,p}(\Omega)=\inf\left\{\dfrac{\int_{\Omega}|\nabla f|^{p}d\Sigma}{\int_{\Omega}|f|^{p}d\Sigma}\,;\,f\in W^{1,p}_{0}(\Omega)\backslash\{0\}\right\},
\end{equation}
where $d\Sigma$ denotes the element volume of $\Sigma^{n}$.

Besides the countless connections between the first eigenvalue and geometric properties, an interesting research theme is determining its value. In the one-dimensional case, this value is well known (see~\cite{Drabek:99}). But, for an arbitrary dimension and $p\neq2$, an exact closed-form expression of the first eigenvalue is not known even for simple domains such as spheres. For this reason, determining upper and lower bounds for this number is interesting. In this direction, in the closed (i.e., compact without boundary) case, many authors have been addressing this problem. For instance, Matei~\cite{Matei:00} obtained a lower estimate for the first nonzero eigenvalue of $p$-Laplacian for closed manifolds with Ricci curvature bounded below by $(n-1)k$, with $k>0$. In the case $k=0$, this sharp estimate was obtained by Valtorta~\cite{Valtorta:12}, and for a general real number $k$, was obtained by Naber and Valtorta~\cite{Valtorta:14} (see also~\cite{Cavalletti:17}). Also, in~\cite{Matei:13}, Matei proved, up to the borderline (i.e., $p$ is equal to the dimension), that the first eigenvalue is bounded from above on each conformal class of Riemannian metrics of unit volume. Moreover, Li and Huang~\cite{Li:20} proved a Cheng-type upper estimate for manifolds with sectional curvature bounded from below. In the context of the isometric immersions, recently, the authors found lower estimates of the first eigenvalue of the $p$-Laplacian in terms of the second fundamental form of such immersion in the unit sphere (see~\cite{Santos:23}).

For complete, noncompact manifolds, the first eigenvalue problem is formulated as follows. Let $\{\Omega_{k}\}$ be an exhaustion of $\Sigma^{n}$ by compact domains, that is, $\{\Omega_{k}\}$ are compact domains such that $\cup_{k=1}^{\infty}\Omega_{k}=\Sigma^{n}$ and $\Omega_{k}\subset\Omega_{k+1}$ for all $k\in\mathbb{N}$. Consider the first eigenvalue $\lambda_{1,p}(\Omega_{k})$ of the following Dirichlet boundary value problem:
\begin{equation}\label{eq_2.4}
\left\{
\begin{array}{ccccc}
\Delta_{p}u &=& -\lambda|u|^{p-2}u, &\quad\mbox{in}\quad \Omega_{k}   \\
u &=& 0, &\quad\mbox{on}\quad \partial\Omega_{k}
\end{array}
\right.
\end{equation}
In~\cite{Veron:91}, Veron showed the existence of the above eigenvalue problem and the variational characterization as in~\eqref{eq_2.3}. On the other hand, Lindqvist~\cite{Lindqvist:12} proved that $\lambda_{1,p}(\Omega_{k})$ is simple for each compact domain $\Omega_{k}$, $k\in\mathbb{N}$. Using the domain monotonicity of $\lambda_{1,p}(\Omega_{k})$, we deduce that $\lambda_{1,p}(\Omega_{k})$ is non-increasing in $k\in\mathbb{N}$ and has a limit which is independent of the choice of the exhaustion of $\Sigma^{n}$. Therefore
\begin{equation}\label{eq_2.5}
\lambda_{1,p}(\Sigma)=\lim_{k\to\infty}\lambda_{1,p}(\Omega_{k}).
\end{equation}
Some authors also call this number the {\em p-fundamental tone} of $\Sigma^{n}$.

In this regard, Evangelista and Seo~\cite{Evangelista:17} gave lower estimates to the $p$-fundamental tone of submanifolds in a Cartan-Hadamard manifold. They obtained lower bounds of the value for geodesic balls and submanifolds with bounded mean curvature and estimates of minimal submanifolds with certain conditions on the norm of the second fundamental form. Very recently, Carvalho and Cavalcante~\cite{Cavalcante:22} obtained a nice lower bound for the $p$-fundamental tone on Riemannian manifolds carrying a special function. As an application, they presented a generalization of McKean's theorem (see~\cite{McKean:70}) and lower estimates of a class of warped product manifolds and a class of Riemannian submersions. Here, we are interested in studying upper estimates for the $p$-fundamental tone of a complete non-compact Riemannian manifold through geometric assumptions on the manifold. Our motivation comes from Gage's remarkable paper published at the beginning of the 80s. In his paper, Gage~\cite{Gage:80} proved an upper estimate for the $2$-Laplacian for metric balls. To do so, Gage's approach involved presenting a result linking the solution of certain differential equations and the Laplace--Beltrami operator.

It is important to highlight that Cheng-type comparison theorems for the first Dirichlet eigenvalue of the $p$-Laplacian were established by Mao \cite{Mao:13} under suitable radial lower bounds on the Ricci curvature. Mao also obtained upper estimates expressed through one-dimensional integrals involving the model warping function. In the hyperbolic setting, his comparison theorem bounds the first eigenvalue of a geodesic ball by that of the corresponding model ball. Our approach follows Gage’s variational construction and yields upper bounds expressed in terms of the dimension, the exponent $p$, the curvature bound, and the radius. 

The outline of the paper is as follows. In Section~\ref{sec:2}, we prove a Gage-type upper estimate for the $p$-fundamental tone of complete noncompact Riemannian manifolds under a lower bound on the Ricci curvature (cf. Theorem~\ref{teo:2.1} and Corollary~\ref{cor:2.2}). In Section~\ref{sec:3}, we present some applications to submanifold theory. We first prove a Leung-type lower estimate for the Ricci curvature of a submanifold of an arbitrary Riemannian manifold, which is sharp for totally
umbilical submanifolds (cf. Lemma~\ref{lemma_4.1}). We then combine this estimate with Corollary~\ref{cor:2.2} to obtain upper bounds for the $p$-fundamental tone
of complete noncompact submanifolds immersed in Riemannian space forms $\mathbb{Q}^m(c)$, $c\in\{-1,0\}$, under suitable bounds on their mean and scalar curvatures (cf. Theorem~\ref{teo_4.1}). Finally, we derive analogous upper estimates for complete noncompact submanifolds immersed in certain product spaces, under a bound on the squared norm of the second fundamental form (cf. Theorem~\ref{teo_4.5} and Corollary~\ref{cor_4.6}).

\section{A Gage type estimate}\label{sec:2}

In this section, we establish an analog of~\cite[Theorem 5.2]{Gage:80} for the $p$-Laplacian, with $p\geq2$. We begin by recalling some preliminary notions from~\cite{Gage:80}.

Let $\Sigma^n$ be a complete Riemannian manifold. We shall consider the following two cases:
\begin{itemize}
\item[(a)] $D=B(q,r_0)$ is a metric ball centered at a point
$q\in\Sigma^n$;
\smallskip

\item[(b)] $D=B(C,r_0)$, where $C\subset\Sigma^n$ is a compact $n$-dimensional body with smooth boundary.
\end{itemize}
In case~(a), let $N=S_q\Sigma$ be the unit sphere in $T_q\Sigma$ and consider the exponential map
\begin{equation}\label{eq:2.1}
F:N\times(0,\infty)\longrightarrow\Sigma^n,\qquad F(x,r)=\exp_q(rx).
\end{equation}
For each $x\in N$, let $R(x)$ denote the cut distance from $q$ in the direction $x$, and define
\begin{equation}\label{eq:2.2}
s(x)=\min\{r_0,R(x)\}.
\end{equation}
In case~(b), let $N=\partial C$ and consider the normal exponential map
\begin{equation}\label{eq:2.3}
F:N\times(0,\infty)\to\Sigma^n,
\end{equation}
where $F(x,r)$ is the point at distance $r$ along the geodesic issuing orthogonally from $\partial C$ at $x$. Let $R(x)$ denote the distance along this geodesic to the first point at which it ceases to minimize the distance to $\partial C$, and again set~\eqref{eq:2.2}. In either case, define
\begin{equation}\label{eq:2.5}
\widetilde N=\{(x,r):x\in N,\ 0<r<s(x)\}.
\end{equation}
For each $r\in(0,r_0)$, set
\begin{equation}\label{eq:2.6}
\widetilde N(r)=\{x\in N:s(x)>r\}.
\end{equation}
Let $J(x,r)$ denote the Jacobian of $F$. If $d\sigma$ denotes the natural volume element on $N$, then
\begin{equation}\label{eq:2.7}
S(r)=\int_{\widetilde N(r)}J(x,r)\,d\sigma
\end{equation}
is the $(n-1)$-dimensional volume of the metric sphere of radius $r$.

Based on these, we establish the following $p$-Laplacian analog of~\cite[Lemma 5.1]{Gage:80}.
\begin{lemma}\label{lem:2.1}
Let $p\geq2$, and assume that $D$ is either the metric ball $D=B(q,r_0)$, as in case {\rm (a)}, or the tubular neighborhood $D=B(C,r_0)$, as in case {\rm (b)}. Let $0<\delta<r_0$, and suppose that $\alpha\in\mathcal{C}([\delta,r_0])$ and
\begin{equation}\label{eq:2.8}
f\in\mathcal{C}([0,r_0])\cap\mathcal{C}^1([\delta,r_0])\cap\mathcal{C}^2((\delta,r_0))
\end{equation}
satisfy the following conditions:
\begin{itemize}
\item[i)] $(|f_r|^{p-2}f_r)_r+\alpha |f_r|^{p-2}f_r+\Gamma f^{p-1}\geq0$, for $\delta<r<r_0;$
\smallskip

\item[ii)] $\Gamma>0$, $f(r_0)=0$ and $f(r)>0$, for $\delta\leq r<r_0$;
\smallskip

\item[iii)] $f_r(\delta)=0$ and $f_r(r)\leq 0$, for $\delta \leq r < r_0$;
\smallskip

\item[iv)] $f$ is constant on $[0,\delta]$;
\smallskip

\item[v)] 
\begin{equation}
\int_{\widetilde N(r)}J_{r}(x,r)\,d\sigma\leq\alpha(r)\int_{\widetilde N(r)}J(x,r)\,d\sigma,\quad\mbox{for}\quad\delta\leq r\leq r_0.
\end{equation}
\end{itemize}
Then $\lambda_{1,p}(D)\leq\Gamma$.
\end{lemma}

\begin{proof}
Let $\rho$ denote the corresponding distance function. More precisely, in case {\rm (a)} we set $\rho(y)=d(q,y)$, whereas in case {\rm (b)} we set $\rho(y)=d(C,y)$.
Define
\begin{equation}\label{eq:2.10}
\varphi(y)=
\begin{cases}
f(\rho(y)), & y\in D,\\
0, & y\in \Sigma\setminus D.
\end{cases}
\end{equation}
Since $\rho$ is Lipschitz and satisfies $|\nabla \rho|=1$ almost everywhere in the corresponding radial region, and since $f(r_0)=0$, the function $\varphi$ has zero trace on $\partial D$. Moreover, by assumption {\rm iv)}, $f$ is constant on $[0,\delta]$, so no singularity of the distance function at the center, or along
$\partial C$ in case {\rm (b)}, creates a difficulty. Hence $\varphi\in W^{1,p}_0(D)$. Furthermore,
\begin{equation}\label{eq:2.11}
|\nabla \varphi|=|f_{r}(\rho)|
\end{equation}
almost everywhere in $D$ in case {\rm (a)}, and almost everywhere in $D\setminus C$ in case {\rm (b)}.

From Rayleigh's characterization of the first Dirichlet eigenvalue,
\begin{equation}\label{eq:2.12}
\lambda_{1,p}(D)\leq\frac{\displaystyle\int_D |\nabla\varphi|^p\,d\Sigma}{\displaystyle\int_D |\varphi|^p\,d\Sigma}.
\end{equation}
By using the radial coordinates determined by the exponential map, and recalling that the cut locus has measure zero, \eqref{eq:2.12} reads
\begin{equation}\label{eq:2.13}
\lambda_{1,p}(D)\leq\frac{\displaystyle\int_N\int_0^{s(x)}|f_{r}(r)|^pJ(x,r)\,dr\,d\sigma
}{\displaystyle\int_N\int_0^{s(x)}f(r)^pJ(x,r)\,dr\,d\sigma}.
\end{equation}
In case {\rm (b)}, the set $C$ contributes zero to the numerator and a nonnegative quantity to the denominator; hence the above inequality remains valid.

Set
\begin{equation}\label{eq:2.14}
Q(r)=|f_{r}(r)|^{p-2}f_{r}(r),
\end{equation}
and since $p\geq2$ and $f\in\mathcal{C}^2((\delta,r_0))$, the function $Q$ is continuously differentiable on $(\delta,r_0)$. Moreover,
\begin{equation}\label{eq:2.15}
Q(r)f_{r}(r)=|f_{r}(r)|^p.
\end{equation}
Fix a point $x\in N$. If $s(x)\leq\delta$, then, by assumption {\rm iv)},
\begin{equation}\label{eq:2.16}
\int_0^{s(x)}|f_{r}(r)|^pJ(x,r)\,dr=0,
\end{equation}
so there is nothing to prove. Assume therefore that $s(x)>\delta$. For $0<\varepsilon<s(x)-\delta$, integration by parts on $[\delta,s(x)-\varepsilon]$ gives
\begin{equation}\label{eq:2.17}
\begin{split}
\int_\delta^{s(x)-\varepsilon}|f_{r}|^pJ\,dr&=\int_\delta^{s(x)-\varepsilon}f_{r}QJ\,dr\\
&=\bigl[fQJ\bigr]_\delta^{s(x)-\varepsilon}-\int_\delta^{s(x)-\varepsilon}fQ_{r}J\,dr
-\int_\delta^{s(x)-\varepsilon}fQJ_r\,dr.
\end{split}
\end{equation}
By assumption {\rm iii)},
\begin{equation}\label{eq:2.18}
Q(\delta)=|f_{r}(\delta)|^{p-2}f_{r}(\delta)=0.
\end{equation}
Furthermore, assumptions {\rm ii)} and {\rm iii)} imply
\begin{equation}\label{eq:2.19}
f\geq0\qquad\text{and}\qquad Q\leq0
\end{equation}
on $[\delta,r_0]$. Since $J(x,r)\geq0$ along minimizing radial geodesics, we have
\begin{equation}\label{eq:2.20}
f(s(x)-\varepsilon)Q(s(x)-\varepsilon)J(x,s(x)-\varepsilon)\leq0.
\end{equation}
Consequently,
\begin{equation}\label{eq:2.21}
\bigl[fQJ\bigr]_\delta^{s(x)-\varepsilon}\leq0,
\end{equation}
and hence
\begin{equation}\label{eq:2.22}
\int_\delta^{s(x)-\varepsilon}|f_{r}|^pJ\,dr\leq-\int_\delta^{s(x)-\varepsilon}fQ_{r}J\,dr-
\int_\delta^{s(x)-\varepsilon}fQJ_r\,dr.
\end{equation}
Since all the integrands are locally integrable on $(\delta,s(x))$, we may let $\varepsilon\to0^{+}$ in~\eqref{eq:2.22} and thus,
\begin{equation}\label{eq:2.23}
\int_\delta^{s(x)}|f_{r}|^pJ\,dr\leq-\int_\delta^{s(x)}fQ_{r}J\,dr-
\int_\delta^{s(x)}fQJ_r\,dr.
\end{equation}
By assumption {\rm i)},
\begin{equation}\label{eq:2.24}
Q_{r}+\alpha Q+\Gamma f^{p-1}\geq0,
\end{equation}
and as $f\geq0$,
\begin{equation}\label{eq:2.25}
-fQ_{r}\leq\alpha fQ+\Gamma f^p.
\end{equation}
Substituting this inequality above, and using the fact that $f_{r}=0$ on $[0,\delta]$, we obtain
\begin{equation*}\label{eq:2.26}
\int_0^{s(x)}|f_{r}|^pJ\,dr\leq\int_\delta^{s(x)}fQ\bigl(\alpha(r)J-J_r\bigr)\,dr+\Gamma\int_0^{s(x)}f^pJ\,dr.
\end{equation*}
We now integrate with respect to $x\in N$. Since $N$ is compact, $r\in[\delta,r_0]$, Fubini's theorem yields
\begin{equation}\label{eq:2.27}
\begin{split}
\int_N\int_0^{s(x)}|f_{r}|^pJ(x,r)\,dr\,d\sigma&\leq\Gamma\int_N\int_0^{s(x)}f^pJ(x,r)\,dr\,d\sigma\\
&\quad+\int_\delta^{r_0}f(r)Q(r)\left[\int_{\widetilde N(r)}\bigl(\alpha(r)J(x,r)-J_r(x,r)\bigr)\,d\sigma\right]dr.    
\end{split}
\end{equation}
By assumption {\rm v)},
\begin{equation}\label{eq:2.28}
\int_{\widetilde N(r)}\bigl(\alpha(r)J-J_r\bigr)\,d\sigma\geq0.
\end{equation}
On the other hand,
\begin{equation}\label{eq:2.29}
f(r)Q(r)=f(r)|f_{r}(r)|^{p-2}f_{r}(r)\leq0.
\end{equation}
Therefore, from~\eqref{eq:2.28} and~\eqref{eq:2.29}, the last term in \eqref{eq:2.27} is nonpositive, and hence
\begin{equation}\label{eq:2.30}
\int_N\int_0^{s(x)}|f_{r}|^pJ(x,r)\,dr\,d\sigma\leq\Gamma\int_N\int_0^{s(x)}f^pJ(x,r)\,dr\,d\sigma.
\end{equation}
Combining this inequality with~\eqref{eq:2.13}, we conclude that
\begin{equation}\label{eq:2.31}
\lambda_{1,p}(D)\leq\Gamma.
\end{equation}
\end{proof}

We also need the following result~\cite[Lemma 4.2]{Gage:80}.
\begin{lemma}\label{lem:2.2}
Let $K$ be a continuous function on $[0,r_0]$. Assume that, along each minimizing radial geodesic $r\mapsto F(x,r)$,
\begin{equation}\label{eq:2.32}
\operatorname{Ric}(\partial_r,\partial_r)\geq(n-1)K(r),
\end{equation}
for every $x\in N$ and every $r\in(0,s(x))$. Let $\varphi\in\mathcal{C}^2([0,r_0])$ be a real-valued function satisfying
\begin{equation}\label{eq:2.33}
\varphi_{rr}+K\varphi=0\quad\mbox{and}\quad\varphi(0)J_r(x,0)-(n-1)\varphi_r(0)J(x,0)\leq0.
\end{equation}
Assume moreover that
\begin{equation}\label{eq:2.35}
\varphi(0)=0\quad\text{ and }\quad\varphi(r)>0\quad\text{for}\quad r\in(0,r_0).
\end{equation}
Then
\begin{equation}\label{eq:2.36}
J_r(x,r)\leq(n-1)\frac{\varphi_r(r)}{\varphi(r)}J(x,r),
\end{equation}
for every $x\in N$ and $r\in(0,s(x))$.
\end{lemma}

Before proving our main result, we recall some basic facts about the generalized trigonometric functions; see, for instance, \cite{Bushell:12}. For $p>1$, define
\begin{equation}\label{eq:2.37}
\pi_p=\frac{2\pi}{p\sin(\pi/p)}.
\end{equation}
Let $L_p:[0,1]\to\left[0,\frac{\pi_p}{2}\right]$
be given by
\begin{equation}\label{eq:2.38}
L_p(y)=\int_0^y \frac{dt}{(1-t^p)^{1/p}}.
\end{equation}
The generalized sine function
\begin{equation}\label{eq:2.39}
\sin_p:\left[0,\frac{\pi_p}{2}\right]\to[0,1]
\end{equation}
is defined as the inverse of $L_p$. In particular, $\pi_p=2L_p(1)$. The generalized cosine function is defined by
\begin{equation}\label{eq:2.40}
\cos_p(r)=\frac{d}{dr}\sin_p(r),\qquad0\leq r\le\frac{\pi_p}{2}.
\end{equation}
The function $\cos_p$ is decreasing on $\left[0,\frac{\pi_p}{2}\right]$ and takes values in $[0,1]$. The functions $\sin_p$ and $\cos_p$ admit extensions to $\mathbb{R}$. We shall use, in particular, the identities
\begin{equation}\label{eq:2.41}
\sin_p(-r)=-\sin_p(r),\qquad\cos_p(-r)=\cos_p(r),
\end{equation}
and
\begin{equation}\label{eq:2.42}
\sin_p\left(\frac{\pi_p}{2}-r\right)=
\sin_p\left(\frac{\pi_p}{2}+r\right),
\end{equation}
\begin{equation}\label{eq:2.43}
\cos_p\left(\frac{\pi_p}{2}-r\right)=-\cos_p\left(\frac{\pi_p}{2}+r\right).
\end{equation}
Moreover, for $0\leq r<\frac{\pi_p}{2}$, we have
\begin{equation}\label{eq:2.44}
\frac{d}{dr}\cos_p(r)=-\sin_p^{p-1}(r)\cos_p^{\,2-p}(r).
\end{equation}

Now, we can prove our main result.
\begin{theorem}\label{teo:2.1}
Let $p\geq 2$, and let $D$ be either the metric ball $D=B(q,r_0)$, as in case {\rm (a)}, or the tubular neighborhood $D=B(C,r_0)$, as in case {\rm (b)}. Assume that
\begin{equation}\label{eq:2.45}
\operatorname{Ric}\geq-(n-1)\beta^2
\end{equation}
on $D$, for some $\beta>0$. Then
\begin{equation}\label{eq:2.46}
\lambda_{1,p}(D)\leq\inf_{0<t<1}\left[\frac{(p-1)^{1/p}\pi_p}{2r_0(1-t)}+\frac{(n-1)\beta}{p}\coth(\beta t r_0)\right]^p.
\end{equation}
\end{theorem}

\begin{proof}
Fix $\delta\in(0,r_0)$, and consider the weighted one-dimensional mixed eigenvalue problem
\begin{equation}\label{eq:2.47}
\begin{cases}
\displaystyle\left(\sinh^{n-1}(\beta r)|f_{r}|^{p-2}f_{r}\right)'+\Lambda_\delta\sinh^{n-1}(\beta r)f^{p-1}=0,
&\delta<r<r_0,\medskip\\
f_{r}(\delta)=0,
\qquad
f(r_0)=0.
\end{cases}
\end{equation}
Since
\begin{equation}\label{eq:2.48}
\sinh^{n-1}(\beta r)>0\quad\text{for}\quad r\in[\delta,r_0],
\end{equation}
this is a regular weighted one-dimensional $p$-Laplacian problem. Its first eigenvalue $\Lambda_\delta$ admits the variational characterization
\begin{equation}\label{eq:2.49}
\Lambda_\delta=\inf_{\substack{u\in W^{1,p}(\delta,r_0)\\u(r_0)=0,\;u\not\equiv0}}
\frac{\displaystyle\int_\delta^{r_0}\sinh^{n-1}(\beta r)|u'(r)|^p\,dr}{\displaystyle
\int_\delta^{r_0}\sinh^{n-1}(\beta r)|u(r)|^p\,dr}.
\end{equation}

Let $f$ be a first eigenfunction associated with $\Lambda_\delta$. We may choose $f$ so that
\begin{equation}\label{eq:2.50}
f(r)>0,\qquad\delta\leq r<r_0.
\end{equation}
From~\eqref{eq:2.47},
\begin{equation}\label{eq:2.51}
\left(\sinh^{n-1}(\beta r)|f_{r}|^{p-2}f_{r}\right)_{r}=-\Lambda_\delta\sinh^{n-1}(\beta r)f^{p-1}<0
\end{equation}
for $\delta<r<r_0$. Since $f_{r}(\delta)=0$, it follows that
\begin{equation}\label{eq:2.52}
\sinh^{n-1}(\beta r)|f_{r}(r)|^{p-2}f_{r}(r)<0
\end{equation}
for every $r\in(\delta,r_0)$. As the weight is strictly positive, we conclude that
\begin{equation}\label{eq:2.53}
f_{r}(r)<0,\qquad\delta<r<r_0.
\end{equation}
Expanding \eqref{eq:2.47} and dividing by $\sinh^{n-1}(\beta r)$, we obtain
\begin{equation}\label{eq:2.54}
\left(|f_{r}|^{p-2}f_{r}\right)_{r}+(n-1)\beta\coth(\beta r)|f_{r}|^{p-2}f_{r}+\Lambda_\delta f^{p-1}=0.
\end{equation}
We now extend $f$ to $[0,r_0]$ by setting
\begin{equation}\label{eq:2.55}
f(r)=f(\delta),\qquad0\leq r\le\delta.
\end{equation}
Therefore $f$ satisfies conditions {\rm i) and iv)} of Lemma~\ref{lem:2.1} with
\begin{equation}\label{eq:2.56}
\alpha(r)=(n-1)\beta\coth(\beta r)\quad\mbox{and}\quad\Gamma=\Lambda_\delta.
\end{equation}
It remains to verify condition {\rm v)} of Lemma~\ref{lem:2.1}. By~\eqref{eq:2.45},
we apply Lemma~\ref{lem:2.2} with
\begin{equation}\label{eq:2.58}
K(r)=-\beta^2\quad\mbox{and}\quad\phi(r)=\frac{1}{\beta}\sinh(\beta r).
\end{equation}
Then
\begin{equation}\label{eq:2.59}
\phi_{rr}(r)-\beta^2\phi(r)=0,\quad\phi(0)=0,\quad\mbox{and}\quad\phi_{r}(0)=1.
\end{equation}

We now verify the initial condition in Lemma~\ref{lem:2.1}. In case {\rm (a)}, since $N=S_q\Sigma$, we have $J(x,0)=0$, and hence
\begin{equation}\label{eq:2.60}
\varphi(0)J_r(x,0)-(n-1)\varphi_r(0)J(x,0)=0.
\end{equation}
In case {\rm (b)}, since $J(x,0)\geq 0$, we have
\begin{equation}\label{eq:2.60.1}
\varphi(0)J_r(x,0)-(n-1)\varphi_r(0)J(x,0)=-(n-1)J(x,0)\leq 0.
\end{equation}
Thus, the initial condition in Lemma~\ref{lem:2.1} is satisfied in both cases. Moreover,
\begin{equation}\label{eq:2.61}
\frac{\phi_{r}(r)}{\phi(r)}=\beta\coth(\beta r).
\end{equation}
Therefore Lemma~\ref{lem:2.2} yields
\begin{equation}\label{eq:2.62}
J_r(x,r)\leq(n-1)\beta\coth(\beta r)J(x,r)
\end{equation}
for every $x\in N$ and every $0<r<s(x)$. Integrating over $\widetilde N(r)$, we obtain
\begin{equation}\label{eq:2.63}
\int_{\widetilde N(r)}J_r(x,r)\,d\sigma\leq(n-1)\beta\coth(\beta r)\int_{\widetilde N(r)}J(x,r)\,d\sigma.
\end{equation}
Hence, the condition {\rm v)} of Lemma~\ref{lem:2.1} is satisfied. Consequently,
\begin{equation}\label{eq:2.64}
\lambda_{1,p}(D)\le\Lambda_\delta.
\end{equation}
We now estimate $\Lambda_\delta$ from above. Set
\begin{equation}\label{eq:2.65}
g(r)=(n-1)\log\sinh(\beta r).
\end{equation}
Then
\begin{equation}\label{eq:2.66}
e^{g(r)}=\sinh^{n-1}(\beta r)\quad\mbox{and}\quad g_{r}(r)=(n-1)\beta\coth(\beta r).
\end{equation}
Define
\begin{equation}\label{eq:2.67}
\theta(r)=\frac{\pi_p(r_0-r)}{2(r_0-\delta)}=\theta_0(r_0-r)\quad\mbox{and}\quad\theta_0
=\frac{\pi_p}{2(r_0-\delta)}.
\end{equation}
Notice that
\begin{equation}\label{eq:2.68}
0\leq\theta(r)\leq\frac{\pi_p}{2},\qquad r\in[\delta,r_0].
\end{equation}
Consider the test function
\begin{equation}\label{eq:2.69}
u(r)=\exp\left(\frac{g(\delta)-g(r)}{p}\right)\sin_p(\theta(r)).
\end{equation}
Since
\begin{equation}\label{eq:2.70}
\theta(r_0)=0\quad\mbox{and}\quad\sin_p(0)=0,
\end{equation}
we have $u(r_0)=0$, so $u$ is admissible in~\eqref{eq:2.49}.

Let
\begin{equation}\label{eq:2.71}
v(r)=\sin_p(\theta(r)).
\end{equation}
Then
\begin{equation}\label{eq:2.72}
u(r)=e^{(g(\delta)-g(r))/p}v(r)\quad\mbox{and}\quad e^{g(r)}|u(r)|^p=e^{g(\delta)}|v(r)|^p.
\end{equation}
Moreover,
\begin{equation}\label{eq:2.73}
u_{r}(r)=e^{(g(\delta)-g(r))/p}\left(v_{r}(r)-\frac{g_{r}(r)}{p}v(r)\right),
\end{equation}
and hence
\begin{equation}\label{eq:2.74}
e^{g(r)}|u_{r}(r)|^p=e^{g(\delta)}\left|v_{r}(r)-\frac{g_{r}(r)}{p}v(r)\right|^p.
\end{equation}
Substituting $u$ into~\eqref{eq:2.49}, we obtain
\begin{equation}\label{eq:2.75}
\Lambda_\delta^{1/p}\leq\frac{\left(\displaystyle\int_\delta^{r_0}\left|v_{r}(r)-\frac{g_{r}(r)}{p}v(r)\right|^pdr\right)^{1/p}}{\left(\displaystyle\int_\delta^{r_0}|v(r)|^pdr
\right)^{1/p}}.
\end{equation}
By Minkowski's inequality,
\begin{equation}\label{eq:2.76}
\begin{split}
\Lambda_\delta^{1/p}&\leq\frac{\left(\displaystyle\int_\delta^{r_0}|v_{r}(r)|^p\,dr
\right)^{1/p}}{\left(\displaystyle\int_\delta^{r_0}|v(r)|^p\,dr\right)^{1/p}}+\frac{\left(\displaystyle\int_\delta^{r_0}\left|\frac{g_{r}(r)}{p}v(r)\right|^pdr
\right)^{1/p}}{\left(\displaystyle\int_\delta^{r_0}|v(r)|^pdr\right)^{1/p}}.    
\end{split}
\end{equation}
We first estimate the first term in~\eqref{eq:2.76}. From~\eqref{eq:2.71}, we have
\begin{equation}\label{eq:2.77}
v_{r}(r)=-\theta_0\cos_p(\theta(r)).
\end{equation}
Thus
\begin{equation}\label{eq:2.78}
|v_{r}|^{p-2}v_{r}=-\theta_0^{p-1}\cos_p^{p-1}(\theta(r)),
\end{equation}
and from~\eqref{eq:2.44}, we obtain
\begin{equation}\label{eq:2.80}
\left(|v_{r}|^{p-2}v_{r}\right)_{r}=-(p-1)\theta_0^p\sin_p^{p-1}(\theta(r)).
\end{equation}
From~\eqref{eq:2.71}, it follows that
\begin{equation}\label{eq:2.81}
\left(|v_{r}|^{p-2}v_{r}\right)_{r}+(p-1)\theta_0^p v^{p-1}=0.
\end{equation}
Furthermore,
\begin{equation}\label{eq:2.82}
\theta(\delta)=\frac{\pi_p}{2}\quad\mbox{and}\quad\theta(r_0)=0,
\end{equation}
so that
\begin{equation}\label{eq:2.83}
v_{r}(\delta)=-\theta_0\cos_p\left(\frac{\pi_p}{2}\right)=0\quad\mbox{and}\quad v(r_0)=0.
\end{equation}
Also,
\begin{equation}\label{eq:2.84}
v(r)>0,\qquad\delta\leq r<r_0.
\end{equation}
Since $v>0$ on $(\delta,r_0)$, the standard nodal characterization of the one-dimensional $p$-Laplacian implies that $v$ corresponds to the first eigenvalue of this mixed problem
\begin{equation}\label{eq:2.85}
\begin{cases}
\displaystyle
\left(|v_{r}|^{p-2}v_{r}\right)_{r}+\lambda v^{p-1}=0,&\delta<r<r_0,\medskip\\
v_{r}(\delta)=0,\qquad v(r_0)=0,
\end{cases}
\end{equation}
corresponding to
\begin{equation}\label{eq:2.86}
\lambda=(p-1)\theta_0^p=(p-1)\left(\frac{\pi_p}{2(r_0-\delta)}\right)^p.
\end{equation}
Multiplying the equation by $v$ and integrating over $[\delta,r_0]$, we obtain
\begin{equation}\label{eq:2.87}
\int_\delta^{r_0}|v_{r}|^p\,dr=(p-1)\theta_0^p\int_\delta^{r_0}v^p\,dr,
\end{equation}
where the boundary term vanishes since $v_{r}(\delta)=0$ and $v(r_0)=0$. Therefore, 
\begin{equation}\label{eq:2.88}
\frac{\displaystyle\int_\delta^{r_0}|v_{r}(r)|^p\,dr}{\displaystyle\int_\delta^{r_0}|v(r)|^p\,dr}=(p-1)\left(\frac{\pi_p}{2(r_0-\delta)}
\right)^p,
\end{equation}
and consequently
\begin{equation}\label{eq:2.89}
\frac{\left(\displaystyle\int_\delta^{r_0}|v_{r}(r)|^p\,dr\right)^{1/p}}{\left(
\displaystyle\int_\delta^{r_0}|v(r)|^p\,dr\right)^{1/p}}=(p-1)^{1/p}\frac{\pi_p}{2(r_0-\delta)}.
\end{equation}
We next estimate the second term in~\eqref{eq:2.76}. Since
\begin{equation}\label{eq:2.90}
g_{r}(r)=(n-1)\beta\coth(\beta r)
\end{equation}
and the function $r\to\coth(\beta r)$ is decreasing on $(0,\infty)$, we have
\begin{equation}\label{eq:2.91}
g_{r}(r)\leq(n-1)\beta\coth(\beta\delta)
\end{equation}
for every $r\in[\delta,r_0]$. Therefore
\begin{equation}\label{eq:2.92}
\frac{\left(\displaystyle\int_\delta^{r_0}\left|\frac{g_{r}(r)}{p}v(r)\right|^pdr
\right)^{1/p}}{\left(\displaystyle\int_\delta^{r_0}|v(r)|^pdr\right)^{1/p}}
\leq\frac{(n-1)\beta}{p}\coth(\beta\delta).
\end{equation}
Combining \eqref{eq:2.76}, \eqref{eq:2.89}, and \eqref{eq:2.92}, we obtain
\begin{equation}\label{eq:2.95}
\Lambda_\delta\leq\left[(p-1)^{1/p}\frac{\pi_p}{2(r_0-\delta)}+\frac{(n-1)\beta}{p}
\coth(\beta\delta)\right]^p.
\end{equation}
Together with \eqref{eq:2.64}, this gives
\begin{equation}\label{eq:2.96}
\lambda_{1,p}(D)\leq\left[(p-1)^{1/p}\frac{\pi_p}{2(r_0-\delta)}+\frac{(n-1)\beta}{p}
\coth(\beta\delta)\right]^p.
\end{equation}
Finally, by taking $\delta=tr_0$, for $0<t<1$ we have $r_0-\delta=r_0(1-t)$, and therefore
\begin{equation}\label{eq:2.98}
\lambda_{1,p}(D)\leq\left[\frac{(p-1)^{1/p}\pi_p}{2r_0(1-t)}+\frac{(n-1)\beta}{p}
\coth(\beta tr_0)\right]^p.
\end{equation}
Since this inequality holds for every $t\in(0,1)$, we conclude that
\begin{equation}\label{eq:2.99}
\lambda_{1,p}(D)\leq\inf_{0<t<1}\left[\frac{(p-1)^{1/p}\pi_p}{2r_0(1-t)}+
\frac{(n-1)\beta}{p}\coth(\beta tr_0)\right]^p.
\end{equation}
\end{proof}

As an immediate consequence of Theorem~\ref{teo:2.1}, we obtain the following estimate for the $p$-fundamental tone.
\begin{corollary}\label{cor:2.2}
Let $p\geq2$, and let $\Sigma^n$ be a complete, connected, noncompact Riemannian manifold. Let $C\subset\Sigma^{n}$ be a compact domain and assume that
\begin{equation}\label{eq:2.100}
\operatorname{Ric}\geq-(n-1)\beta^2\quad\text{on}\quad\Sigma^{n}\setminus C
\end{equation}
for some $\beta>0$. Then
\begin{equation}\label{eq:2.101}
\lambda_{1,p}(\Sigma)\leq\left(\frac{(n-1)\beta}{p}\right)^p.
\end{equation}
\end{corollary}

\begin{proof}
Fix $q_0\in\Sigma^{n}$, and as $C$ is compact, there exists $A>0$ such that $C\subset B(q_0,A)$. Since $\Sigma^{n}$ is complete and noncompact, the Hopf-Rinow theorem
implies that $\Sigma^{n}$ is unbounded. Hence there exists a sequence $\{q_j\}_{j\in\mathbb{N}}\subset\Sigma^{n}$ such that $d(q_j,q_0)\to+\infty$. 

Set $d_j=\operatorname{dist}(q_j,C)$, and for every $x\in C$,
\begin{equation}\label{eq:2.102}
d(q_j,x)\geq d(q_j,q_0)-d(q_0,x)\geq d(q_j,q_0)-A.
\end{equation}
Therefore
\begin{equation}\label{eq:2.103}
d_j\geq d(q_j,q_0)-A.
\end{equation}
and so $d_j\to\infty$. Let $R_j=\frac{d_j}{2}$. If $x\in B(q_j,R_j)$, then
\begin{equation}\label{eq:2.104}
\operatorname{dist}(x,C)\geq\operatorname{dist}(q_j,C)-d(x,q_j)>d_j-R_j=\frac{d_j}{2}>0.
\end{equation}
Hence $B(q_j,R_j)\subset\Sigma^{n}\setminus C$. In particular,
\begin{equation}\label{eq:2.105}
\operatorname{Ric}\geq-(n-1)\beta^2\quad\mbox{on}\quad B(q_j,R_j).
\end{equation}
By the variational characterization of the $p$-fundamental tone, every function in $W^{1,p}_0(B(q_j,R_j))$, extended by zero to $\Sigma^{n}$, is an admissible test function on $\Sigma^{n}$. Consequently,
\begin{equation}\label{eq:2.106}
\lambda_{1,p}(\Sigma)\leq\lambda_{1,p}(B(q_j,R_j)).
\end{equation}
Fix $t\in(0,1)$. by applying Theorem~\ref{teo:2.1} to $B(q_j,R_j)$, we obtain
\begin{equation}\label{eq:2.107}
\lambda_{1,p}(\Sigma)\leq\left[\frac{(p-1)^{1/p}\pi_p}{2R_j(1-t)}+\frac{(n-1)\beta}{p}
\coth(\beta tR_j)\right]^p.
\end{equation}
Since $R_j\to\infty$, for every fixed $t\in(0,1)$,
\begin{equation}\label{eq:2.108}
\frac{(p-1)^{1/p}\pi_p}{2R_j(1-t)}\to0\quad\mbox{and}\quad\coth(\beta tR_j)\longrightarrow1.
\end{equation}
Passing to the limit as $j\to\infty$ gives
\begin{equation}\label{eq:2.109}
\lambda_{1,p}(\Sigma)\leq\left(\frac{(n-1)\beta}{p}\right)^p.
\end{equation}
\end{proof}

\section{Applications to submanifolds}\label{sec:3}
			
Let $\overline{M}^m$ be an arbitrary  Riemannian manifold and let $\Sigma^{n}$ be a submanifold immersed isometrically in $\overline{M}^m$. Fix a point $q\in \Sigma^{n}$ and take a local orthonormal frame $\{e_{1},\ldots,e_{m}\}$ of $\overline{M}^m$ around $q$ such that $\{e_{1},\ldots,e_{n}\}$ are tangent fields and $\{e_{n+1},\ldots,e_{m}\}$ are normal fields on $\Sigma^{n}$. For each $\alpha$, $n+1\leq\alpha\leq m$, we define a linear map $A_{\alpha}:T_{q}\Sigma\to T_{q}\Sigma$ by
\begin{equation}\label{eq_4.1}
\langle A_{\alpha}(X),Y\rangle=\langle\overline{\nabla}_{X}Y,e_{\alpha}\rangle,
\end{equation}
where $X,Y$ are tangent fields on $\Sigma^{n}$ and $\overline{\nabla}$ is the Riemannian connection of $\overline{M}^{m}$. The square $S$ of the norm of the second fundamental form $A$, the mean curvature vector $h$ and the mean curvature function $H$ are defined by:
\begin{equation}\label{eq_4.2}
S=\sum_{\alpha}{\rm tr}(A_{\alpha}^{2}),\quad h=\dfrac{1}{n}{\rm tr}(A)\quad\mbox{and}\quad H=|h|.
\end{equation}

It is well known that the curvature tensor of $\Sigma^{n}$ can be described in terms of its second fundamental form $A$ and of the curvature tensor of $\overline{M}^m$. In particular, the Ricci curvature of $\Sigma^{n}$ is given by
\begin{equation}\label{eq_4.3}
{\rm Ric}(X,X)=\sum_{i}\langle\overline{R}(X,e_{i})X,e_{i}\rangle+n\langle A_{h}(X),X\rangle-\sum_{\alpha}|A_{\alpha}(X)|^{2},
\end{equation}
for every tangent vector field $X$ and $\{e_{1},\ldots,e_{n}\}$ denotes a local orthonormal frame field on $\Sigma^{n}$. By contracting~\eqref{eq_4.3}, we have the scalar curvature $R$ of $\Sigma^{n}$
\begin{equation}\label{eq_4.4}
R=\sum_{i,j}\langle\overline{R}(e_{j},e_{i})e_{j},e_{i}\rangle+n^{2}H^{2}-S.
\end{equation}

To prove our main results, we need the following Leung-type lower estimate of the Ricci curvature tensor (cf.~\cite[Main Theorem]{Leung:92}).
\begin{lemma}\label{lemma_4.1}
Let $\Sigma^n$ be a submanifold of the Riemannian manifold $\overline{M}^m$. Then
\begin{equation}\label{eq_4.5}
\begin{split}
{\rm Ric}(X,X)&\geq-\dfrac{n-1}{n}\left(S+\frac{n(n-2)}{\sqrt{n(n-1)}}H\sqrt{S-nH^{2}}-2nH^2\right)|X|^{2}\\
&\quad+\sum_{i}\langle\overline{R}(X,e_{i})X,e_{i}\rangle,
\end{split}
\end{equation}
for all $X\in\mathfrak{X}(\Sigma)$.
\end{lemma}

\begin{proof}
For each $\alpha$, define the linear map $\phi_{\alpha}:T_{q}\Sigma\to T_{q}\Sigma$ by
\begin{equation}\label{eq_4.6}
\langle\phi_{\alpha}(X),Y\rangle=\langle A_{\alpha}(X),Y\rangle-\langle h,e_{\alpha}\rangle\langle X,Y\rangle,
\end{equation}
and a bilinear map $\phi:T_{q}\Sigma\times T_{q}\Sigma\to(T_{q}\Sigma)^{\perp}$ by
\begin{equation}\label{eq_4.7}
\phi(X,Y)=\sum_{\alpha}\langle\phi_{\alpha}(X),Y\rangle e_{\alpha}
\end{equation}
It is easy to check that each map $\phi_{\alpha}$ is traceless and that
\begin{equation}\label{eq_4.8}
|\phi|^{2}=S-nH^{2}.
\end{equation}
Observe that $|\phi|^{2}=0$ if and only if $\Sigma^{n}$ is a totally umbilic submanifold of $\overline{M}^m$.

In terms of~\eqref{eq_4.7}, \eqref{eq_4.1} can be rewritten as
\begin{equation}\label{eq_4.9}
\begin{split}
{\rm Ric}(X,X)&=\sum_{i}\langle\overline{R}(X,e_{i})X,e_{i}\rangle+(n-2)\langle\phi_{h}(X),X\rangle\\
&\quad+(n-1)H^{2}|X|^{2}-\sum_{\alpha}|\phi_{\alpha}(X)|^{2}.
\end{split}
\end{equation}
Since $\phi$ is traceless, from the Cauchy-Schwarz inequality, we have
\begin{equation}\label{eq_4.10}
\langle\phi_{h}(X),X\rangle\geq-|\langle\phi_{h}(X),X\rangle|\geq-\sqrt{\dfrac{n-1}{n}}|\phi_{h}||X|^{2}
\end{equation}
and
\begin{equation}\label{eq_4.11}
\sum_{\alpha}\langle\phi^{2}_{\alpha}(X),X\rangle\leq\dfrac{n-1}{n}\sum_{\alpha}|\phi_{\alpha}|^{2}|X|^{2}=\dfrac{n-1}{n}|\phi|^{2}|X|^{2}.
\end{equation}
Besides this, we observe that
\begin{equation}\label{eq_4.12}
\phi_{h}=\sum_{\alpha}\langle h,e_{\alpha}\rangle\phi_{\alpha}.
\end{equation}
From Cauchy-Schwarz's inequality and Hilbert-Schmidt's norm definition, we have
\begin{equation}\label{eq_4.13}
|\phi_{h}|^{2}=\bigg|\sum_{\alpha,i}\langle h,e_{\alpha}\rangle\phi_{\alpha}\bigg|^{2}\leq\left(\sum_{\alpha}\langle h,e_{\alpha}\rangle^{2}\right)\left(\sum_{\alpha}|\phi_{\alpha}|^{2}\right)\leq H^{2}|\phi|^{2}.
\end{equation}
Hence
\begin{equation}\label{eq_4.14}
\langle\phi_{h}(X),X\rangle\geq-\sqrt{\dfrac{n-1}{n}}H|\phi||X|^{2},
\end{equation}
and consequently,
\begin{equation}\label{eq_4.15}
\begin{split}
{\rm Ric}(X,X)&\geq\sum_{i}\langle\overline{R}(X,e_{i})X,e_{i}\rangle-\dfrac{(n-1)(n-2)}{\sqrt{n(n-1)}}H|\phi||X|^{2}\\
&\quad+(n-1)H^{2}|X|^{2}-\dfrac{n-1}{n}|\phi|^{2}|X|^{2},
\end{split}
\end{equation}
for every tangent vector field $X$. By replacing~\eqref{eq_4.8} in~\eqref{eq_4.15} we finish.
\end{proof}

For the first application of Corollary~\ref{cor:2.2}, we will consider the case where the ambient space is a Riemannian manifold with constant sectional curvature, that is $\overline{M}^m=\mathbb{Q}^{m}(c)$, with $c\in \{0,-1\}$. Combining this identity with Lemma~\ref{lemma_4.1}, we obtain the following result.
\begin{theorem}\label{teo_4.1}
Let $\Sigma^n$ be a complete noncompact immersed submanifold of the Riemannian space form $\mathbb{Q}^{m}(c)$ $(c\in\{0,-1\}\,\,\mbox{and}\,\,n\geq 2)$. If the mean curvature of $\Sigma^{n}$ satisfies $H\leq b$ and the scalar curvature of $\Sigma^{n}$ satisfies $R \geq-n(n-1)a^2$ for real numbers $a\geq0$ and $b^{2}+c\geq0$. Then
\begin{equation}\label{eq_4.16}
\lambda_{1,p}(\Sigma)\leq\dfrac{(n-1)^p}{p^p}\left((n-2)(b^2+c)+(n-2)b\sqrt{a^2+b^{2}+c}+(n-1)a^2\right)^\frac{p}{2}.
\end{equation}
\end{theorem}

\begin{proof}
Since $\mathbb{Q}^{m}(c)$ has constant sectional curvature $c$, Lemma~\ref{lemma_4.1} becomes
\begin{equation}\label{eq_4.17}
{\rm Ric}(X,X)\geq-\dfrac{n-1}{n}\left(S+\frac{n(n-2)}{\sqrt{n(n-1)}}H\sqrt{S-nH^{2}}-nc-2nH^2\right)|X|^{2},
\end{equation}
for all $X\in\mathfrak{X}(\Sigma)$. Also, the scalar curvature of $\Sigma^{n}$ satisfies
\begin{equation}\label{eq_4.18}
R=n(n-1)c+n^{2}H^{2}-S.
\end{equation}
Thus,
\begin{equation}\label{eq_4.19}
\begin{split}
S+\frac{n(n-2)}{\sqrt{n(n-1)}}&H\sqrt{S-nH^2}-cn-2nH^2\\
&=\frac{n(n-2)}{\sqrt{n(n-1)}}H\sqrt{n(n-1)c+n(n-1)H^{2}-R}\\
&\quad+n(n-2)(H^{2}+c)-R.
\end{split}
\end{equation}
Now, by the assumptions on $H$ and $R$,
\begin{equation}\label{eq_4.20}
\begin{split}
\frac{n(n-2)}{\sqrt{n(n-1)}}H&\sqrt{n(n-1)c+n(n-1)H^{2}-R}+n(n-2)(H^{2}+c)-R\\
&\leq n(n-1)a^2+\frac{n(n-2)}{\sqrt{n(n-1)}}b\sqrt{n(n-1)(a^2+b^{2}+c)}+n(n-2)(b^2+c)\\
&=n(n-1)a^2+n(n-2)b\sqrt{a^2+b^{2}+c}+n(n-2)(b^2+c).
\end{split}
\end{equation}
Hence, \eqref{eq_4.5} yields
\begin{equation}
\operatorname{Ric}(X,X)\geq-(n-1)A|X|^2,
\end{equation}
where
\begin{equation}
A:=(n-1)a^2+(n-2)b\sqrt{a^2+b^2+c}+(n-2)(b^2+c).
\end{equation}
So, if $A>0$, by taking $\beta=\sqrt{A}$ and applying Corollary~\ref{cor:2.2}, we get
\begin{equation}
\lambda_{1,p}(\Sigma)\leq\left(\frac{(n-1)\beta}{p}\right)^p=\frac{(n-1)^p}{p^p}A^{p/2}.
\end{equation}
On the other hand, if $A=0$, then $\operatorname{Ric}\geq0$. Hence, for every $\varepsilon>0$,
\begin{equation}
\operatorname{Ric}\geq-(n-1)\varepsilon^2.
\end{equation}
Hence, by applying Corollary~\ref{cor:2.2} with $\beta=\varepsilon$ and letting
$\varepsilon\to0$, we obtain
\begin{equation}
\lambda_{1,p}(\Sigma)=0.
\end{equation}
Therefore, in both cases, the asserted estimate follows.
\end{proof}

The next two results concern submanifolds immersed in Euclidean space.
\begin{corollary}\label{cor_4.1}
If $\Sigma^n$ is a complete noncompact minimal submanifold in $\mathbb{R}^{m}$ such that the scalar curvature satisfies $R\geq-n(n-1)a^2$ for some real number $a\geq0$, then
\begin{equation}\label{eq_4.22}
\lambda_{1,p}(\Sigma) \leq \frac{(n-1)^{\frac{3p}{2}} a^p}{p^p} .
\end{equation}
\end{corollary}

\begin{proof}
The result follows directly from Theorem~\ref{teo_4.1} by taking $c=0$ and $b=0$.
\end{proof}

\begin{corollary}\label{cor_4.2}
Let $\Sigma^n$ be a complete noncompact submanifold in $\mathbb{R}^{m}$ with nonnegative scalar curvature. If the mean curvature of $\Sigma^{n}$ satisfies $H\leq b$ for some real number $b$, then
\begin{equation}\label{eq_4.23}
\lambda_{1,p}(\Sigma)\leq\frac{2^{\frac{p}{2}}(n-1)^{p}(n-2)^{\frac{p}{2}} b^p}{p^p}.
\end{equation}
\end{corollary}

\begin{proof}
The result follows directly from Theorem~\ref{teo_4.1} by taking $c=0$ and $a=0$.
\end{proof}

\begin{remark}
In the particular case $n=2$, Corollary~\ref{cor_4.2} yields
\begin{equation}
\lambda_{1,p}(\Sigma)=0.
\end{equation}
Indeed, in this case, the factor $(n-2)^{p/2}$ vanishes. Equivalently, in Theorem~\ref{teo_4.1}, by taking $a=0$ and $c=0$, we have $A=0$, and hence the argument in the degenerate case gives $\lambda_{1,p}(\Sigma)=0$.
\end{remark}

For the last application, let us consider the product smooth manifold $M^m\times\mathbb{R}$ endowed with the metric
\begin{equation}\label{eq_4.24}
\langle v,w\rangle_{(p,t)}=\langle(\pi_\mathbb{R})_*v, (\pi_\mathbb{R})_*w\rangle_{\mathbb{R}}+\langle(\pi_M)_*v,(\pi_M)_*w\rangle_{M},\quad(p,t)\in M^m\times\mathbb{R},
\end{equation}
$v,w\in T_{(p,t)}( M^m\times\mathbb{R})$, where $\pi_{M}$ and $\pi_{\mathbb{R}}$ denote the projections onto the corresponding factors. Such a space is called a {\em product space}, and in what follows we shall write $ M^m\times\mathbb{R}$ to denote it. Associated with the product structure, we have the following unit vector field
\begin{equation}\label{eq_4.25}
\partial_{t}:=(\partial/\partial_{t})\big|_{(p,t)},\quad(p,t)\in M^m\times\mathbb{R}
\end{equation}

Now, let $\Sigma^{n}$ be a submanifold of $M^{m}\times\mathbb{R}$. Along $\Sigma^n$, the vector field $\partial_t$ decomposes as
\begin{equation}\label{eq_4.26}
\partial_{t}=T+N,
\end{equation}
where $T$ and $N$, denote the tangential and normal part of $\partial_{t}$, respectively. From this,
\begin{equation}\label{eq_4.27}
|T|^{2}+|N|^{2}=1.
\end{equation}

Hence, we have the following theorem.
\begin{theorem}\label{teo_4.5}
Let $M^m\times\mathbb{R}$ be a product space whose Riemannian factor $M^m$ has sectional curvature bounded from below by $c$. Let $\Sigma^n$ be a complete noncompact minimal submanifold of $M^m\times\mathbb{R}$ such that $S\leq b$ for some $b\geq 0$. Then
\begin{equation}
\lambda_{1,p}(\Sigma)
\leq\frac{(n-1)^p}{p^p}\left(\frac{b}{n}\right)^{p/2},\qquad \text{if}\quad c\geq 0,
\end{equation}
and
\begin{equation}
\lambda_{1,p}(\Sigma)\leq\frac{(n-1)^p}{p^p}\left(\frac{b}{n}-c\right)^{p/2},
\qquad\text{if}\quad c<0.
\end{equation}
\end{theorem}

\begin{proof}
Let $\{e_1,\ldots,e_n\}$ be a local orthonormal frame on $\Sigma^n$. Denoting by $K_M$ the sectional curvature of $M^m$, we have
\begin{equation}
\langle \overline{R}(X,e_i)X,e_i\rangle=K_M(X^*,e_i^*)|X^*\wedge e_i^*|_M^2,
\end{equation}
where
\begin{equation}
X^*=X-\langle X,\partial_t\rangle\partial_t\quad\mbox{and}\quad e_i^*=e_i-\langle e_i,\partial_t\rangle\partial_t.
\end{equation}
A direct computation gives
\begin{equation}
\sum_{i=1}^n |X^*\wedge e_i^*|_M^2=(n-1-|T|^2)|X|^2-(n-2)\langle X,T\rangle^2.
\end{equation}
Consequently, since $K_M\geq c$,
\begin{equation}
\sum_{i=1}^n
\langle \overline{R}(X,e_i)X,e_i\rangle\geq
\begin{cases}
0, & c\geq 0,\medskip\\
c(n-1)|X|^2, & c<0.
\end{cases}
\end{equation}
Since $\Sigma^n$ is minimal, $H=0$. Hence, by Lemma~\ref{lemma_4.1} and the assumption $S\leq b$, if $c\geq0$ we obtain
\begin{equation}
\operatorname{Ric}(X,X)\geq-(n-1)\frac{b}{n}|X|^2,
\end{equation}
whereas, if $c<0$,
\begin{equation}
\operatorname{Ric}(X,X)\geq-(n-1)\left(\frac{b}{n}-c\right)|X|^2.
\end{equation}
Suppose first that $c\geq0$ and $b>0$. By taking $\beta^2=\frac{b}{n}$, Corollary~\ref{cor:2.2} gives
\begin{equation}
\lambda_{1,p}(\Sigma)\leq\frac{(n-1)^p}{p^p}\left(\frac{b}{n}\right)^{p/2}.
\end{equation}
If $b=0$, then $\operatorname{Ric}\geq0$. Applying Corollary~\ref{cor:2.2} with an arbitrary $\varepsilon>0$ in place of $\beta$ and letting $\varepsilon\to0$, we obtain
\begin{equation}
\lambda_{1,p}(\Sigma)=0.
\end{equation}
If $c<0$, then $\frac{b}{n}-c>0$. Taking $\beta^2=\frac{b}{n}-c$ and applying Corollary~\ref{cor:2.2}, we obtain
\begin{equation}
\lambda_{1,p}(\Sigma)\leq\frac{(n-1)^p}{p^p}\left(\frac{b}{n}-c\right)^{p/2}.
\end{equation}
\end{proof}

We close this section with the following consequence.
\begin{corollary}\label{cor_4.6}
Let $\Sigma^{n}$ be a complete noncompact minimal submanifold of $\mathbb{H}^{m}\times\mathbb{R}$. If $S\leq b$ for some constant $b>0$, then
\begin{equation}
\lambda_{1,p}(\Sigma)\leq\dfrac{(n-1)^{p}}{p^{p}}\left(\dfrac{b}{n}+1\right)^{p/2}.
\end{equation}
\end{corollary}

\begin{proof}
The result follows directly from Theorem~\ref{teo_4.5} by taking $c=-1$.
\end{proof}

\section*{Acknowledgments}
The first author is partially supported by CNPq, Brazil, grant 303311/2025-8 and Propesqi (UFPE). The second author is partially supported by CNPq, Brazil.

\end{document}